\documentclass[english]{amsart}

\usepackage{mathrsfs}
\usepackage{mathtools}
\usepackage{graphicx}
\usepackage{array}
\usepackage{amsmath}
\usepackage{amssymb}
\usepackage{amsthm}
\usepackage{amsfonts}
\usepackage{xcolor}
\usepackage[pagebackref,hyperindex=true,colorlinks=true,linkcolor=blue,citecolor=biblio,urlcolor=biblio]{hyperref}

\usepackage{cleveref}
\crefname{subsection}{subsection}{subsections}

\usepackage{tikz,tkz-tab}
\usetikzlibrary{matrix,cd,arrows.meta,fit,calc,positioning,decorations.pathreplacing,decorations.pathmorphing}
\tikzcdset{
arrow style=tikz,
diagrams={>={Stealth[scale=0.8]}}
}

\usepackage[shortlabels]{enumitem}
\setlist[enumerate]{label=\rm{(\arabic*)}}

\theoremstyle{plain}
\newtheorem{thm}{Theorem}[section]
\newtheorem{thmIntro}{Theorem}

\newtheorem*{thm*}{Theorem}

\newtheorem{prop}[thm]{Proposition}

\newtheorem*{prop*}{Proposition}
\newtheorem{lemma}[thm]{Lemma}

\newtheorem{pb}{Problem}
\newtheorem*{algo*}{Algorithm}

\theoremstyle{definition}

\newtheorem{ex}[thm]{Example}
\newtheorem{exIntro}[pb]{Example}
\newtheorem{defi}[thm]{Definition}

\newtheorem{rmq}[thm]{Remark}

\numberwithin{equation}{section}

\DeclareMathOperator{\hght}{height}
\DeclareMathOperator{\lgth}{length}

\DeclareMathOperator{\Fitt}{Fitt}

\DeclareMathOperator{\Proj}{Proj}

\DeclareMathOperator{\car}{char}

\newcommand*{\mrm}[1]{\mathrm{#1}}
\newcommand*{\mb}[1]{\mathbb{#1}}
\newcommand*{\mbf}[1]{\mathbf{#1}}
\newcommand*{\mc}[1]{\mathcal{#1}}

\newcommand*{\pn}[1]{\mathbb{P}^{#1}}
\renewcommand*{\k}{\mrm{k}}
\newcommand*{\PI}{\pn{}(I)}

\newcommand*{\Tor}{\mb{T}}
\newcommand*{\IPI}{I_{\PI}}

\newcommand*{\R}{\mc{R}}
\newcommand*{\Sy}{\mc{S}}
\newcommand*{\nd}{\mathfrak{d}}
\renewcommand{\d}[2]{d_{#2}(#1)}
\newcommand*{\A}{\mathcal{A}}

\date{\today}
\title[Polar Cremona maps of arbitrarily large degree]{Plane polar Cremona maps of arbitrarily large degree in positive characteristic}

\author{R\'emi Bignalet-Cazalet}
\address{Universit\`a degli studi di Genova, Dipartimento di matematica,
Via Dodecaneso 35, 16146 Genova (GE), Italy }
\email{bignalet@dima.unige.fr}
\thanks{The author was funded by the European Union's Horizon 2020 research and innovation program as a  Marie Sk\l{}odowska-Curie fellow of the Istituto Nazionale di Alta Matematica "Francesco Severi" grant No. 713485.
}


\begin{document}
\definecolor{biblio}{rgb}{0,0.65,1}
\definecolor{xdxdff}{rgb}{0.49,0.49,1}
\definecolor{ttttff}{rgb}{0.2,0.2,1}
\definecolor{zzzzff}{rgb}{0.6,0.6,1}
\definecolor{indigo}{rgb}{0.29,0,0.51}
\definecolor{veronese}{rgb}{0.35,0.4,0.13}

\begin{abstract}
A result of I.V.\ Dolgachev states that the complex homaloidal polynomials in three variables, i.e.\ the complex homogeneous polynomials whose polar map is birational, are of degree at most three. In this note we describe homaloidal polynomials in three variables of arbitrarily large degree in positive characteristic. Using combinatorial arguments, we also classify line arrangements whose polar map is homaloidal in positive characteristic.

\end{abstract}
\maketitle
\section*{Introduction}\label{labIntro}
Given a homogeneous polynomial $f\in\k[x_0,\ldots,x_m]$ over a field $\k$, the \emph{polar map} $\Phi_f:\pn{m}_\k\dashrightarrow \pn{m}_\k$ of $f$ is the rational map defined by the linear system $\langle \frac{\partial f}{\partial x_0},\ldots,\frac{\partial f}{\partial x_m}\rangle$. The polynomial $f$ is called \emph{homaloidal} if $\langle \frac{\partial f}{\partial x_0},\ldots,\frac{\partial f}{\partial x_m}\rangle$ has no fixed component and $\Phi_f$ is birational.

It was established by I.V.\ Dolgachev \cite[Theorem 4]{dolgachev2000polar} that if $\k=\mb{C}$, the homaloidal polynomials in three variables are either of degree $2$, defining a smooth conic in the projective plane, or of degree $3$, defining either a union of three lines in general position or a union of a smooth conic with one of its tangents. These polynomials remain homaloidal when the base field $\k$ has characteristic  greater than $2$. This leads to the following question which is a generalisation of \cite[Question 3.7]{dorHassSim2012polar}.
\begin{pb}\label{pbDolgClass}
Over a field $\k$ of positive characteristic, are
there other homaloidal polynomials than the ones in Dolgachev's classification?
\end{pb}
In \cite[Proposition 4.6]{Big2018TorSymAlg}, a first example of a homaloidal polynomial of degree $5$ over a field of characteristic $3$ was produced, answering both \Cref{pbDolgClass} and \cite[Question 3.7]{dorHassSim2012polar}. Very recently, the following example of a homaloidal curve of degree $5$ in characteristic $3$ was also described.

\begin{exIntro}\label{exIntro}
In characteristic $3$, the polynomial $f=x_0(x_1^2+x_0x_2)(2x_1^2+x_0x_2)$, whose zero locus is the union of two conics intersecting with multiplicity two in two distinct points with the tangent at one intersection point, is homaloidal.
\end{exIntro}

The negative answer to \Cref{pbDolgClass} leads to the following question.

\begin{pb}\label{pbCar}
Does there exist homaloidal polynomials in three variables of arbitrary large degree over fields of arbitrarily large characteristic?
\end{pb}

In this note, we answer this question, see \Cref{thmConsGen} for our expanded result.  Following the designation in \cite{Hirzebruch1983ArrLinesAndHyp} we say that a union of $n$ distinct lines through a given point $z_0$ with another line not passing through $z_0$ is a \emph{near-pencil} arrangement of $n+1$ lines. In addition, given a reduced projective curve $F=\mb{V}(f)$ which is the zero locus $\mb{V}(f)$ in $\pn{2}_\k$ of a homogeneous polynomial $f$, we say that $F$ is \emph{homaloidal} if $f$ is homaloidal.
\begin{thmIntro}\label{thmGen}
Let $\k$ be a field of characteristic $p$ and let $n\in\mathbb{N}_{>0}$ be a multiple of $p$. Then the near-pencil arrangement of $n+1$ lines is homaloidal.
\end{thmIntro}
This result provides an answer to \Cref{pbCar}. For instance, let $n\in\mb{N}_{>0}$ be such that $n\equiv 0\mod 5$ and let $\k$ be a field of characteristic $5$. Then the polynomial $f_n$ in \Cref{thmGen} has degree $n+1$ and is homaloidal. Moreover, given a prime number $p$, a field $\k$ of characteristic $p$ and a positive integer $m$, \Cref{thmGen} gives a homaloidal polynomial of degree $mp+1$, so homaloidal polynomials exist in arbitrarily large degree and in any prime characteristic.

Remark that the \emph{base ideal} of the polar map $\Phi_f$, i.e.\ the ideal generated by the partial derivatives of a polynomial $f$ defines the singular locus of the curve $\mb{V}(f)$ defined by $f$ in $\pn{2}_\k$. In this direction, the proof given by I.V.Dolgachev about the classification of homaloidal complex polynomials relies on the \emph{Jung-Milnor's formula} over $\mb{C}$ relating several invariants of singularities \cite[Lemma 3]{dolgachev2000polar}. In contrast, our proof of \Cref{thmGen} relies on the study of the torsion of the symmetric algebra of the base ideal of $\Phi_{f}$, an approach that fits in line with previous works such as \cite{RussoSimis2001OnBirMap}, \cite{dorHassSim2012polar}, and \cite{Big2018TorSymAlg}. We emphasize that most of the polynomials we consider in this note define free curves (a curve being \emph{free} if, by definition, the base ideal of the polar map is determinantal \cite[Def 2.1]{dimca2015freenessversus}). We specially focus on this case  since, when the base ideal has a linear syzygy, free curves are the curves whose singular schemes have maximal length \cite[Cor 1.2]{dimca2015freenessversus}.

\subsection*{Contents of the paper}
In the first section, we recall the relations between the symmetric algebra of the base ideal of a rational map $\Phi$ and the graph of $\Phi$. Birationality of a
map can be checked via its graph which explains the strategy of detecting a
birational map via the symmetric algebra of its base ideal.

The second section constitutes the heart of our work. As a central idea, one can study the reduction modulo $p$ of the presentation matrix of the base ideal in order to predict a drop of the \emph{topological degree} of the polar map, see \Cref{subSecRedModp}. The next step is then to evaluate this drop. We carry on this evaluation by describing the generic fibre of the \emph{naive graph} and thus describing the generic fibre of the graph itself. This implies in particular that the polynomials $f_n$ are homaloidal (\Cref{descGenTor}). We end this section by providing another example of a polynomial of degree $5$ which is homaloidal in characteristic $3$. Its zero locus in $\pn{2}_k$ is  the union of the unicuspidal rampho\"id quartic and the tangent cone at its cusp (\Cref{exChar3}).

In the third section, we focus on line arrangements (that is plane curves which are union of lines) and we show that, over an algebraically closed field of characteristic $p>0$, the only homaloidal line arrangements are the ones defining unions of three general lines or near-pencil of $lp+1$ lines for any $l\geqslant 1$, see \Cref{classifLineArr}. This classification follows from the description of the singularities defined by line arrangements.

The explicit computations given in this paper were made using basic functions of the software systems \textsc{Polymake} and  \textsc{Macaulay2} with the \textsc{Cremona} package  \cite{stagliano2017Mac2Pack} associated. The corresponding codes are available on request.

\subsection*{Acknowledgements} I warmly thank Adrien Dubouloz and Daniele Faenzi for useful remarks and suggestions about early versions of this note. I also thank deeply the anonymous referee for having raised the question of the classification of homaloidal line arrangements and for having pointed out a mistake about the computation of Milnor number of singularities in a previous version of this paper.

\section{Graph and naive graph}\label{sectionSymRees}

In this note, all the fields are assumed to be algebraically closed and denoted by the same letter $\k$.

\subsection{Multidegree of a subscheme of $\pn{2}\times\pn{2}$, projective degrees}
For this presentation, we follow \cite[7.1]{Dolg2011ClassAlgGeo}. Given $l\in\lbrace 0,1,2\rbrace$, we denote by $H^l$ a general codimension $l$ linear subspace of $\pn{2}_\k$. For a subscheme $X\subset\pn{2}_\k\times\pn{2}_\k$ of codimension $2$, the \emph{multidegree} $\big(\d{X}{0},\d{X}{1},\d{X}{2}\big)$ is defined by : \begin{equation}\label{multInter}
\d{X}{i}=\lgth\big(X\cap p^{-1}_1(H^{i})\cap p^{-1}_2(H^{2-i})\big)
\end{equation} where $p_1:\pn{2}_\k\times\pn{2}_\k\rightarrow\pn{2}_\k$ and $p_2:\pn{2}_\k\times\pn{2}_\k\rightarrow\pn{2}_\k$ are the first and second projection respectively.

Consider now a rational map $\Phi=(\phi_0:\phi_1:\phi_2):\pn{2}_\k\dashrightarrow \pn{2}_\k$ with base ideal $I=(\phi_0,\phi_1,\phi_2)\subset R=\k[x_0,x_1,x_2]$ where $\phi_0,\phi_1,\phi_2$ are homogeneous polynomials of the same degree that do not share any common factor. The \emph{graph} $\Gamma$ of $\Phi$ is by definition the closure of $\lbrace(x,\Phi(x)),\; x\in\pn{2}_\k\backslash \mb{V}(I)\rbrace\subset\pn{2}_\k\times\pn{2}_\k$ in the Zariski topology. It is an irreducible variety of codimension $2$. We define the \emph{projective degrees} $\d{\Phi}{0},\d{\Phi}{1},\d{\Phi}{2}\big)$ of $\Phi$ as the terms in the multidegree $\big(\d{\Gamma}{0},\d{\Gamma}{1},\d{\Gamma}{2}\big)$ of $\Gamma$.

Since $\Phi$ is birational if and only if $\d{\Phi}{0}=1$, this last quantity has a special importance and is called the \emph{topological degree} of $\Phi$.

\subsection{Decomposition of the naive graph} Let $\Phi=(\phi_0:\phi_1:\phi_2):\pn{2}_\k\dashrightarrow\pn{2}_\k$ be a rational map with base ideal $I=(\phi_0,\phi_1,\phi_2)$ in the coordinate ring $R=\k[x_0,x_1,x_2]$ of $\pn{2}_\k$ and let $\R(I):=\oplus_{i\geq 0}I^it^i\subset R[t]$ be the \emph{Rees algebra} of $I$. By \cite[7.1.3]{dolgachev2000polar} the blow-up $\Proj\big(\R(I)\big)$ of $\pn{2}_\k$  with respect to $I$ is the graph $\Gamma$ of $\Phi$. Moreover, $\R(I)$ is the epimorphic image of the symmetric algebra $\Sy(I)$ of $I$ via the epimorphisms $I^{\otimes i}\twoheadrightarrow I^i$. Hence, the ideal of the graph $\Gamma\subset \pn{2}_\k\times\pn{2}_\k$ of $\Phi$ contains the ideal of $\PI=\Proj\big(\Sy(I)\big)\subset \pn{2}_\k\times\pn{2}_\k$ generated by the entries of the matrix $\begin{pmatrix}
y_0 & y_1 & y_2
\end{pmatrix}M_S$ where $S=R[y_0,y_1,y_2]$ stands for the coordinate ring of $\pn{2}_\k\times\pn{2}_\k$ and $M_S$ stands for a presentation matrix of $I\subset R$ tensored with $S$  \cite[Subsection 1.1]{Big2018TorSymAlg} (note that we use the same notation $M$ and $M_S$ from now on). The ideal $I$ is said to be of \emph{linear type} if $\Gamma=\PI$ \cite[1.1 Ideals of linear type]{Vasconcelos2005Int}. 

\begin{defi}
The \emph{naive graph} of $\Phi$ is the projectivization $\PI=\Proj\big(\Sy(I)\big)$ of the symmetric algebra of $I$.
\end{defi}

\begin{ex}\label{exFond} Let $\Phi=(\phi_0:\phi_1:\phi_2):\pn{2}_\k\dashrightarrow \pn{2}_\k$ be a dominant rational map with base ideal $I=(\phi_0,\phi_1,\phi_2)\subset\k[x_0,x_1,x_2]$ of height at least $2$ and where $\phi_0,\phi_1,\phi_2$ have the same degree $d$. Note that, in this setting, the condition $\hght(I)\geq 2$ is equivalent to the fact that $(\phi_0,\phi_1,\phi_2)$ do not share any common factor, a property which we always assume in the following. Assume moreover that $\Phi$ is \emph{determinantal}, i.e.\ that the polynomials $\phi_0,\phi_1,\phi_2$ are the $2$-minors of a given $3\times 2$-matrix such that all its entries in the first column are homogeneous of degree $a$ and all its entries in the second column are homogeneous of degree $b$ (hence $a+b=d)$. The Hilbert-Burch theorem \cite[Theorem 20.15]{eisenbud1995algebra} implies that $I$ has a free resolution of the form:
\begin{equation*}
\begin{tikzcd}[row sep=0.8em,column sep=2em,minimum width=2em]
0\ar{r}&R^2\ar[r, "{M}"]&  R^3  \ar[rr,"{(\phi_0 \;\phi_1\;\phi_2)}"]&& I \ar{r}& 0.
\end{tikzcd}
\end{equation*} where $M$ is a $3\times 2$-matrix with entries in $R$. Hence $M$ is a presentation matrix of $I$ and $\PI$ is the intersection of two divisors of $\pn{2}_\k\times\pn{2}_\k$ of bidegree $(a,1)$ and $(b,1)$ respectively. Considering the case where $I$ is of linear type, since $\Gamma$ has codimension $2$, $\PI=\Gamma$ is a complete intersection and the projective degrees of $\Phi$ are given by B\'ezout's theorem:
\begin{equation}\label{MultiDegGraphNaif} \begin{cases}
&\d{\Phi}{0}=\d{\PI}{0}=ab\\
&\d{\Phi}{1}=\d{\PI}{1}=a+b\\
&\d{\Phi}{2}=\d{\PI}{2}=1.
\end{cases}
\end{equation}
\end{ex}
Assume that $\Phi:\pn{2}_\k\dashrightarrow\pn{2}_\k$ is a determinantal rational map with base ideal $I\subset R=\k[x_0,x_1,x_2]$ of height $2$ and denote by $M$ the $3\times 2$-presentation matrix of $I$. Then the complement subscheme $\Tor=\overline{\PI\backslash \Gamma}$ of $\Gamma$ in $\PI$ is supported on \begin{equation}\label{topTors}\underset{\mbf{x}\in\mb{V}\big(\Fitt_2(I)\big)}{\bigcup}\lbrace \mbf{x}\rbrace \times \pn{2}_\k\end{equation} where $\Fitt_2(I)$ is the second Fitting ideal of $I$ \cite[Corollary-Definition 20.4]{eisenbud1995algebra}, by definition generated by the entries of $M$, see \cite[Corollary 1.4]{Big2018TorSymAlg} for a reference.
\begin{defi}
Let $\Phi=(\phi_0:\phi_1:\phi_2):\pn{2}_\k\dashrightarrow \pn{2}_\k$ be a determinantal rational map with base ideal $I=(\phi_0,\phi_1,\phi_2)$ of height $2$. The \emph{naive projective degrees} $\nd_0(\Phi),\nd_1(\Phi),\nd_2(\Phi)$ of $\Phi$ are defined by the multidegree $\big(\d{\PI}{0},\d{\PI}{1},\d{\PI}{2}\big)$ of its naive graph $\PI$. 
\end{defi} 

\begin{ex}\label{exDegTop} If $I$ is not necessarily of linear type in \Cref{exFond}, we still have that the naive projective degree of $\Phi$ are $(ab,a+b,1)$ because $\PI$ is still a complete intersection. However, if there is an extra part $\mb{T}=\overline{\PI\backslash\Gamma}$ in $\PI$ with support as in \eqref{topTors}, we have then \[ \big(\d{\PI}{0},\d{\PI}{1},\d{\PI}{2}\big)=\big(\d{\Gamma}{0},\d{\Gamma}{1},\d{\Gamma}{2}\big)+\big(\d{\Tor}{0},\d{\Tor}{1},\d{\Tor}{2}\big).\] So the topological degree $\d{\Phi}{0}$ is strictly smaller than $\nd_0(\Phi)$ because $\d{\Tor}{0}\geq \lgth\big(\mb{V}\big(\Fitt_2(I)\big)$ is non zero. The quantity $\d{\Gamma}{0}=\d{\PI}{0}-\d{\Tor}{0}$ depends moreover on the scheme structure of $\Tor$ and is the object of \Cref{subSecWeigth}.
\end{ex}

\section{Contribution of the torsion}
In this section, following the situation described in \Cref{exDegTop}, we illustrate first on an example how to estimate the drop of the topological degree in positive characteristic compared to characteristic $0$. We analyse then in greater generality how this modification impacts the computation of the topological degree of $\Phi$.

\subsection{Reduction of the presentation matrix modulo p}\label{subSecRedModp} In what follows, for a homogeneous ideal $I$ of $R=\k[x_0,x_1,x_2]$ and an integer $t$, we denote by $I_t$ the homogeneous piece of $I$ of degree $t$. Let $n\in\mb{N}_{>1}$ and let $F_n$ be the union of $n$ distinct lines through a point $z_0\in\pn{2}_\k$ with any other line not passing through $z_0$. We can reduce to the situation where $\mb{V}(x_2)$ is the latter line  and two lines among the $n$\textsuperscript{th} firsts are $\mb{V}(x_0)$ and $\mb{V}(x_1)$ so that $z_0=(0:0:1)$. We can consequently assume without loss of generality that an equation of $F_n$ reads $f_n=x_0x_1l_2\cdots l_{n-1}x_2$ where, for all $i\in\lbrace 2,\ldots,n-1\rbrace$, $l_i$ belongs to $(x_0,x_1)_1$ (set $f_n=x_0x_1x_2$ if $n=2$). The ideal $I$ of partial derivatives of $f_n$ is then equals to \begin{align*}
I=(x_1l_2\cdots l_{n-1}x_2+&x_0x_1\frac{\partial}{\partial x_0}(l_2\cdots l_{n-1})x_2,\\ &x_0l_2\cdots l_{n-1}x_2+x_0x_1\frac{\partial}{\partial x_1}(l_2\cdots l_{n-1})x_2,x_0x_1l_2\cdots l_{n-1}).
\end{align*} 
\begin{lemma}\label{presQ} A minimal presentation matrix of $I$ reads \[M=\begin{pmatrix}
x_0 & 0 \\ x_1 & x_1l_2\cdots l_{n-1} \\-nx_2 & -l_2\cdots l_{n-1}x_2-x_1\frac{\partial}{\partial x_1}(l_2\cdots l_{n-1})x_2\end{pmatrix}
\]
\end{lemma}

\begin{proof}
The ideal $I$ has depth $2$ for otherwise the first two generators of $I$ would be divisible by either $x_0$, $x_1$ or $l_i$ for $i\in\lbrace 2,\ldots,n-1\rbrace$ which is excluded by the assumption that all the lines in $F_n$ are distinct. A direct computation shows that \[\begin{cases} x_1x_2l_2\ldots l_{n-1}+x_0x_1x_2\frac{\partial}{\partial x_0}(l_2\ldots l_{n-1})=M_1\\
x_0x_2l_2\ldots l_{n-1}+x_0x_1x_2\frac{\partial}{\partial x_1}(l_2\ldots l_{n-1})=M_2\\
x_0x_1l_2\ldots l_{n-1}=M_3\end{cases}\] where, given $j\in\lbrace 1,2,3\rbrace$, $M_j$ stands for $( -1 )^j$ times the minor obtained from $M$ by leaving out the $j$\textsuperscript{th} row (in order to check these equalities, remark that, for any $i\in \lbrace 2,\ldots,n-1\rbrace$, $l_i=x_0\frac{\partial l_i}{x_0}+x_1\frac{\partial l_i}{x_1}$). Hence $I$ is a determinantal ideal given by the $2$-minors of $M$. Since it has the expected depth, the Hilbert-Burch theorem asserts that a free resolution of $I$ reads:
\begin{equation*}
\begin{tikzcd}[row sep=3em,column sep=0.5cm,minimum width=2em]
  0 \ar{r}& R^2\ar[r,"{M}"]& R^3 \ar{r}&  I \ar{r}&0.
\end{tikzcd}
\end{equation*}
Moreover, since it does not have constant entries, $M$ is a minimal presentation matrix of $I$.
\end{proof}

\begin{prop}\label{propNaiveDegree}
Let $n\in\mb{N}_{>1}$ and $\k$ be an algebraically closed field such that $p=\car(\k)$ does not divide $n$. Then $I$ is of linear type and the polar map $\Phi_{f_n}$ of $f_n$ has multidegree $(n-1,n,1)$.
\end{prop}
\begin{proof}
Since $M$ has one column of linear entries and one column of entries of degree $n$, the naive multidegree is $(n-1,n,1)$, see \Cref{exFond}. Moreover $\Fitt_2(I)=(x_0,x_1,x_2)$ is not supported on any point of $\pn{2}_\k$ so, by \eqref{topTors}, $I$ is of linear type. Hence the graph $\Gamma$ and the naive graph $\PI$ coincide, so the projective degrees and the naive projective degrees of $\Phi_{f_n}$ coincide. 
\end{proof}
We consider now the case where $\car(\k)$ divides $n$ in a more general situation.

\subsection{Weight of the torsion}\label{subSecWeigth}
In all the section, we put $n\in\mb{N}_{>1}$, $\k$ be any field unless otherwise specified, $R=\k[x_0,x_1,x_2]$ and we consider the ideal $I$ generated by the $2$-minors of the matrix \[M=\begin{pmatrix}
\lambda_0 & p_0 \\ \lambda_1 & p_1 \\ \lambda_2 & p_2
\end{pmatrix}
\] with entries in $R=\k[x_0,x_1,x_2]$ such that for all $j\in\lbrace 0,1,2\rbrace$, $\lambda_j$ belongs to $(x_0,x_1)_1$ and $p_j$ belongs to $(x_0,x_1)^{n-2}_{n-1}$, the homogeneous piece of degree $n-1$ of the ideal $(x_0,x_1)^{n-2}$. We assume moreover that $I$ has height $2$ and that there exists $j\in\lbrace 0,1,2\rbrace$ such that $p_j\in (x_0,x_1)^{n-2}_{n-1}\backslash (x_0,x_1)^{n-1}_{n-1}.$

In $S=R[y_0,y_1,y_2]$, consider now the ideal \[ \IPI=(\lambda_0y_0+\lambda_1y_1+\lambda_2y_2,y_0p_0+y_1p_1+y_2p_2)\] of the embedding of $\PI$ in $\pn{2}_{\mbf{x}}\times\pn{2}_{\mbf{y}}$ generated by the entries of the matrix $\begin{pmatrix}
y_0&y_1&y_2
\end{pmatrix}M$. Following the computation in \Cref{exFond}, $\PI$ being a complete intersection of two divisors of bidegree $(1,1)$ and $(n-1,1)$, one has \[\big(\d{\PI}{0},\d{\PI}{1},(\d{\PI}{2}\big)=(n-1,n,1).\] Moreover since all the entries of $M$ are in the ideal $(x_0,x_1)$, the radical $\sqrt{\Fitt_2(I)}$ of the ideal $\Fitt_2(I)$ of entries of $M$ is contained in $(x_0,x_1)$. Actually, since $I$ has height $2$ and the polynomials in the first column of $M$ are linear in $x_0$ and $x_1$ one has $\Fitt_2(I)=(x_0,x_1)$. Hence, as previously stated in \eqref{topTors}, the naive graph $\PI=\mb{V}(\IPI)\subset\pn{2}_{\mbf{x}}\times\pn{2}_{\mbf{y}}$ is the union of a torsion part $\Tor$ supported on $\mb{V}(x_0,x_1)=\lbrace (0:0:1)\rbrace\times\pn{2}_\k$, and of the graph $\Gamma=\overline{\PI\backslash \mb{V}(x_0,x_1)}$ of the map $\Phi$ whose base ideal is by the $2$-minors ideal $I$ of $M$. The next result is a consequence of \cite[Theorem 5.14]{BuCiDAnd2018MultiGrad} but we will give a self-contained proof.

\begin{lemma}\label{descGenTor}\label{tteCar} Under the previous conditions on $M$, the torsion component $\Tor$ of $\PI$ has multidegree $(n-2,0,0)$ and the graph $\Gamma$ of $\Phi$ has multidegree $(1,n,1)$.
\end{lemma}

\begin{proof}
We analyse separately each element of the multidegree $\big(\d{\Gamma}{0},\d{\Gamma}{1},\d{\Gamma}{2}\big)$.

\begin{itemize}
\item \textbf{Case $i=2$.} Take a general point $\mbf{y}\in \pn{2}_{\k}$ as the intersection of two general lines $H_1$, $H_2$ of $\pn{2}_\k$ and consider the intersection \[\PI_{\mbf{y}}=\PI\cap p^{-1}_2(H_1)\cap p^{-1}_2(H_2)\subset\mb{P}^2_{\mbf{y}}.\]
Under our assumptions, this intersection is a complete intersection of a line and a curve of degree $n-1$ in $\mb{P}^2_{\mbf{y}}$. Moreover, since $\lambda_j\in (x_0,x_1)_1$ for all $j\in\lbrace 0, 1,2\rbrace$ and since there exists $j\in\lbrace 0, 1,2\rbrace$ such that $p_j\in  (x_0,x_1)^{n-2}_{n-1}\backslash (x_0,x_1)^{n-1}_{n-1}$, this complete intersection decomposes as the union of the point $\mb{V}(x_0,x_1)_\mbf{y}\in \mb{P}^2_{\mbf{y}}$ with multiplicity $n-2$ and of another point with multiplicity $1$. By the generality assumption on $\mbf{y}$, we can assume that the subscheme $\PI\cap \pn{2}_\mbf{y}$ is defined by the ideal $\big(x_1, x_0^{n-2}(x_0+\alpha)\big)$ for some $\alpha\in\k\backslash\lbrace 0\rbrace$. Since $\Gamma=\overline{\PI\backslash \mb{V}(x_0,x_1)}$ is defined by the saturation $[I_{\PI}:(x_0:x_1)^{\infty}]$ of the ideal $I_{\PI}$ of $\PI$ by the ideal $(x_0,x_1)=\Fitt_2(I)$, see \eqref{topTors}, the only points of $\pn{2}_k$ over which the fibre of $\PI$ contributes to $\d{\Gamma}{0}$ are those different from $\mb{V}(x_0,x_1)$. Thus the other point is the only element in $\Gamma\cap p^{-1}_2(H_1)\cap p^{-1}_2(H_2)$. Hence \[\d{\Gamma}{0}=\lgth \big(\Gamma \cap p^{-1}_2(H_1)\cap p^{-1}_2(H_2)\big)=1\] and \[\d{\Tor}{0}=\d{\PI}{0}-\d{\Gamma}{0}=n-2.\]
\item \textbf{Case $i=1$.} Since $I$ has height $2$, the linear system defined by the $2$-minors of $M$ does not have fixed components so $\d{\Gamma}{1}=\d{\PI}{1}=n$ and thus $\d{\Tor}{1}=0$.
\item \textbf{Case $i=0$.} The torsion component $\Tor$ being supported over $\mb{V}(x_0,x_1)$, the intersection $\Tor \cap p^{-1}_1(H_1)\cap p^{-1}_1(H_2)$ of $\Tor$ with inverse images of general lines in $\pn{2}_\k$ is empty so $\d{\Tor}{2}=0$ and
$\d{\Gamma}{2}=\d{\PI}{2}-\d{\Tor}{2}=1.$
\end{itemize}
To sum up, $\Tor$ has multidegree $(n-2,0,0)$ and $\Gamma$ has multidegree \[(n-1,n,1)-(n-2,0,0)=(1,n,1).\]
\end{proof}

We have the following extension of \Cref{thmGen}.

\begin{thm}\label{thmConsGen}
\begin{enumerate}
\item\label{thmConsGen2} Let $n\in\mb{N}_{>1}$ and assume that $p=\car\k$ divides $n$, then the near-pencil arrangements  of $n+1$ lines is homaloidal.
\item\label{thmConsGen1} Let $n\in\mb{N}_{>1}$ and assume that $p=\car\k$ divides $n(n-1)-1$, then the curve $G_n=\mb{V}\big(x_0x_1(x_1^{n-1}+x_0^{n-2}x_2)\big)$ is homaloidal.
\end{enumerate}
\end{thm}

\begin{proof}
\begin{enumerate}
\item[\ref{thmConsGen2}] By \Cref{presQ}, a presentation matrix of the ideal $I$ of partial derivatives of $f_n$ verifies the conditions of \Cref{tteCar}. Hence $\Phi_{f_n}$ is birational and since the associated linear system has no fixed component, the polynomial $f_n$ is homaloidal.

\item[\ref{thmConsGen1}] Let $n\in\mb{N}_{>1}$. The ideal \[I=(x_1^{n}+(n-1)x_0^{n-2}x_1x_2, nx_0x_1^{n-1}+x_0^{n-1}x_2,x_0^{n-1}x_1)\] of partial derivatives of $g_n=x_0x_1(x_1^{n-1}+x_0^{n-2}x_2)$ has presentation matrix \[M=\begin{pmatrix}
nx_0 & 0 \\ -x_1 & x_0^{n-2}x_1 \\ -(n(n-1)-1)x_2 & -nx_1^{n-1}-x_0^{n-2}x_2
\end{pmatrix},\] Indeed, $I$ has height $2$ for otherwise $x_0$ or $x_1$ would divide $x_1^{n}+(n-1)x_0^{n-2}x_1x_2$ and $nx_0x_1^{n-1}+x_0^{n-1}x_2$ which is not the case. Moreover \[\begin{cases} x_1^{n}+(n-1)x_0^{n-2}x_1x_2=M_1\\
nx_0x_1^{n-1}+x_0^{n-1}x_2=M_2\\
x_0^{n-1}x_1=M_3,\end{cases}\] where given $j\in\lbrace 1,2,3\rbrace$, $M_j$ is equal to $( -1 )^j$ times the minor obtained from $M$ by leaving out the $j$\textsuperscript{th} row. Hence $I$ is a determinantal ideal and, by application of Hilbert-Burch theorem, $M$ is a minimal presentation matrix of $I$. Now, if $p$ divides $n(n-1)-1$, the matrix $M$ verifies the conditions of \Cref{tteCar}. So, in this case, $\Phi_{g_n}$ is birational and $g_n$ is homaloidal.

\end{enumerate}
\end{proof}
\begin{rmq}
The method of reduction modulo $p$ we just described also applies to \Cref{exIntro} and to the quintic $Q_5=\mb{V}\big(x_0(x_1^2+x_0x_2)(x_1^2+x_0x_2+x_0^2)\big)$ described in \cite{Big2018TorSymAlg}.
\end{rmq}
\subsection{Limits and perspectives} The fact that the presentation matrix of the jacobian ideal reduces well modulo $p$ does not always occur, as illustrated by the following example.

\begin{ex}\label{exChar3}
Let $h=x_2(x_1^4-2x_0x_1^2x_2+x_0^2x_2^2-x_1x_2^3)\in \k[x_0,x_1,x_2]$. Its zero locus in $\pn{2}_\k$ is the union of the unicuspidal rampho\"id quartic with the tangent cone at its cusp, see \cite{Moe2008RatCusCur}. Over a field $\k$ of characteristic $0$, a computation with \textsc{Macaulay2} shows that a presentation matrix of the ideal of partial derivatives of $h$ reads:
\[
\begin{pmatrix}
15x_1^2+3x_0x_2 & 72x_0x_1+15x_2^2 \\
8x_1x_2     &    2x_1^2+30x_0x_2\\
-2x_2^2      &   -8x_1x_2   
\end{pmatrix}.
\] We can a priori not expect to apply \Cref{tteCar} after reduction modulo $p$. However, after reducing modulo $3$, a presentation matrix of the reduction of $I$ modulo $3$ reads
\[
\begin{pmatrix}
 0  &  x_1^3-x_0x_1x_2-x_2^3 \\
 x_1  & x_0x_2^2    \\
 -x_2 & -x_1x_2^2 
\end{pmatrix}.
\]
This implies that the polar map of $h$ is birational by \Cref{tteCar} (here, remark that the torsion is supported on $\mb{V}(x_1,x_2)$ and that the maximal power of $x_0$ is $1$ is the second column). By application of Hilbert-Burch theorem, we also have that the induced linear system does not have fix components so $h$ is actually homaloidal.
\end{ex}

\begin{rmq}
As pointed out by \Cref{exChar3} and \Cref{thmConsGen1} of \Cref{thmConsGen}, the classification of homaloidal plane curves in any characteristic seems to be a challenging problem, especially by only looking to the reduction modulo $p$ of the syzygies of the jacobian ideal. One can however restrict first to the classification of homaloidal line arrangements and this is the object of next section.
\end{rmq}

\section{Classification of homaloidal line arrangements in positive characteristic}
As a guideline for the section, let us state first our result about the classification of homaloidal line arrangements.

\begin{prop}\label{classifLineArr} Given an algebraically closed field $\k$ of characteristic $p>0$, the only homaloidal line arrangements are:
\begin{enumerate}[label=(\roman*)]
\item\label{item3lines} the union of three general lines,
\item\label{itemNearPencil} the near-pencils of $n+1$ lines where $p$ divides $n$.
\end{enumerate}
\end{prop}

Our proof of \Cref{classifLineArr} mainly relies on the observation that, as far as the topological degree of the polar map of an arrangement is concerned, the only quantity to consider is the numbers of lines defining the singularities of the arrangement. More precisely, a singularity $z$ of a line arrangement $\A=\mb{V}(f)$ being the intersection of $r\geqslant 2$ lines of $\A$, the numerical contribution of $z$ in the computation of $\d{\Phi_f}{0}$ only depends on whether the characteristic $p$ divides $r$ or not, see \Cref{lemmaMultSingLineArr} for the precise result. Given this fact, the proof of \Cref{classifLineArr} aims to characterize combinatorially near-pencils of $p+1$ lines among all arrangements of $p+1$ lines and this combinatorial characterization follows from \cite[Th.1]{BruijnErdos1948CombiPb}.

In the following, given an integer $d\geqslant 4$, we let $f=l_1\cdots l_d$ be the product of $d$ homogeneous linear polynomials $l_1,\ldots,l_d\in\k[x_0,x_1,x_2]$ and $\A=\mb{V}(f)$ be the line arrangement defined by $f$. Moreover, using the designation in \cite{Hirzebruch1983ArrLinesAndHyp}, a point $z$ in the singular locus of $\A$ which is the intersection point of $r$ lines is called a \emph{$r$-fold point}.

The field $\k$ being algebraically closed, the topological degree $\d{\Phi_f}{0}$ of $\Phi_f$ is the degree of the fiber of a generic point of $\pn{2}$, that is:
\[\d{\Phi_f}{0}=\deg \mb{V}(I^g:I^{\infty})\] where $I=(\frac{\partial f}{\partial x_0},\frac{\partial f}{\partial x_1},\frac{\partial f}{\partial x_2})$ is the jacobian ideal of $f$, $I^g= (a\frac{\partial f}{\partial x_0}+b\frac{\partial f}{\partial x_1}+c\frac{\partial f}{\partial x_2},\alpha\frac{\partial f}{\partial x_0}+\beta\frac{\partial f}{\partial x_1}+\gamma\frac{\partial f}{\partial x_2})$ is the ideal defined by two generic linear combinations of $\frac{\partial f}{\partial x_0},\frac{\partial f}{\partial x_1},\frac{\partial f}{\partial x_2}$ and $I^g:I^{\infty}$ stands for the saturation ideal of $I^g$ by $I$, see \cite[7.1.3]{Dolg2011ClassAlgGeo} for this computation of the topological degree. Since $\mb{V}(I^g:I^{\infty})$ is set-theoretically equal to $\mb{V}(I^g)\backslash\mb{V}(I)$, one has thus: \begin{equation}
\label{eqTopDegSat}\d{\Phi_f}{0}=(d-1)^2-\underset{z\in\mb{V}(I)}{\sum}m_z
\end{equation} where $m_z$ is the multiplicity of $z$ in the scheme $\mb{V}(I^g)$ (note that this latter expression of $\d{\Phi_f}{0}$ is true for any reduced plane curve $\mb{V}(f)$ and not only for line arrangements). Over the field of complex numbers $\mathbb{C}$, by \cite[4.2]{dimca2017HypArr}, given an $r$-fold point $z\in\mb{V}(I)$ one has \[m_z=\mu_{f,z}=(r-1)^2\] where $\mu_{f,z}$ stands for the local Milnor number of $\A$ at $z$, see \cite[Definition 2.17]{dimca2017HypArr} for the definition of Milnor numbers. Over a field of positive characteristic, the relation between $\d{\Phi_f}{0}$ and Milnor numbers of the singularities of $\A$ is much blurred, in particular because Milnor number is not an invariant under contact equivalence anymore (see \cite{HefezRodriguesSalomao2019SingPosChar} for the definition of contact equivalence and more precision about the definition of Milnor number in positive characteristic). In other words, over a field $\k$ of positive characteristic, \Cref{eqTopDegSat} is still valid by definition but the numbers $m_z$ cannot be interpreted as the Milnors numbers of the singularities defined by $f$ (even if we won't need it, let us however precise that the numbers $m_z$ appeared to be related to the \emph{Milnor number of a hypersurface} $\mu(\mathcal{O}_f)$, a contact equivalent invariant defined in \cite[end of section 3]{HefezRodriguesSalomao2019SingPosChar}. We also point out that typical behaviors of singularities in positive characteristic prevent the classification of homaloidal polynomials via Dolgachev's approach in \cite[Lemma 3]{dolgachev2000polar} over $\mb{C}$, see \cite{MellWall2001PencilsOfCurves} and \cite{Duc2016InvPosChar} for instances of such behaviors when reducing modulo $p$).

\begin{lemma}\label{lemmaMultSingLineArr}
Let $p=\car(\k)$ and $z$ an $r$-fold point of a line arrangement $\A=\mb{V}(f)$, $f=l_1\cdots l_d$. Denote by $m_z$ the multiplicity of $z$ in $\mb{V}(I^g)$ as in \Cref{eqTopDegSat}:
\begin{enumerate}
\item if $p$ divides $r$, then $m_z=(r-1)^2+(r-2)$,
\item\label{labelLemmaMultSing} if $p$ does not divide $r$, then $m_z=(r-1)^2$.
\end{enumerate}
\end{lemma}

\begin{proof}
To describe $m_z$, we first explain why it is enough to make the computation in the case that $\A$ is a near-pencil of $r+1$ lines such that $z$ is the intersection point of $r$ lines. Once we have our local model for $z$, we use \Cref{thmConsGen} to compute $m_z$.

Let $\A=\mb{V}(f)$, $f=l_1\cdots l_d$, write $z=(0:0:1)$ by choosing coordinates of $\pn{2}_{\k}$ and label the linear polynomials $l_1,\ldots,l_d$ defining $\A$ such that $l_1,\ldots,l_r\in (x_0,x_1)$ and $l_{r+1},\ldots, l_d\in (x_0,x_1)^c$, $(x_0,x_1)^c$ being the complementary of the ideal $(x_0,x_1)$.

Now, by applying the elementary rules of derivations, one has that a generic linear combination $a\frac{\partial f}{\partial x_0}+b\frac{\partial f}{\partial x_1}+c\frac{\partial f}{\partial x_2}$ of $\frac{\partial f}{\partial x_0},\frac{\partial f}{\partial x_1},\frac{\partial f}{\partial x_2}$ reads:
\begin{align*}
a\frac{\partial f}{\partial x_0}+b\frac{\partial f}{\partial x_1}&+c\frac{\partial f}{\partial x_2} = (\underset{j=1}{\overset{r}{\sum}}(l_1\cdots l_{j-1}(a\frac{\partial l_j}{\partial x_0}+b\frac{\partial l_j}{\partial x_1})l_{j+1}\cdots l_r)(l_{r+1}\cdots l_d)\\
+& (l_1\cdots l_r)(\underset{j=r+1}{\overset{d}{\sum}}(l_{r+1}\cdots l_{j-1}(a\frac{\partial l_j}{\partial x_0}+b\frac{\partial l_j}{\partial x_1}+c\frac{\partial l_j}{\partial x_2})l_{j+1}\cdots l_
d)
\end{align*}
In the localization $\k[x_0,x_1,x_2]_{(x_0,x_1)}$ of $\k[x_0,x_1,x_2]$ at the prime $(x_0,x_1)$, remark that $(\underset{j=r+1}{\overset{d}{\sum}}(l_{r+1}\cdots l_{j-1}(a\frac{\partial l_j}{\partial x_0}+b\frac{\partial l_j}{\partial x_1}+c\frac{\partial l_j}{\partial x_2})l_{j+1}\cdots l_d)$ is a unit since $(a:b:c)\in\pn{2}_\k$ is generic.

Hence denoting $\overline{v}\in \k[x_0,x_1,x_2]_{(x_0,x_1)}$ the localization of $v\in\k[x_0,x_1,x_2]$ at the prime $(x_0,x_1)$, one has:
\begin{align*}
\overline{a\frac{\partial f}{\partial x_0}+b\frac{\partial f}{\partial x_1}+c\frac{\partial f}{\partial x_2}}=u\times &\big( \overline{(\underset{j=1}{\overset{r}{\sum}}(l_1\cdots l_{j-1}(a'\frac{\partial l_j}{\partial x_0}+b'\frac{\partial l_j}{\partial x_1})l_{j+1}\cdots l_r)l_{r+1}}\\
&\overline{+c'(l_1\cdots l_r)}\big)
\end{align*}
where $u$ is a unit of $\k[x_0,x_1,x_2]_{(x_0,x_1)}$. In other words, $(I^g)_{(x_0,x_1)}$ is equal to  $(a'\frac{\partial f'}{\partial x_0}+b'\frac{\partial f'}{\partial x_1}+c'\frac{\partial f'}{\partial x_2},\alpha'\frac{\partial f'}{\partial x_0}+\beta'\frac{\partial f'}{\partial x_1}+\gamma'\frac{\partial f'}{\partial x_2})_{(x_0,x_1)}$ where \[f'=l_1'\cdots l_r'l_{r+1}'\] such that $l_1'\cdots l_r'\in (x_0,x_1)$ are distinct lines passing by $z=(0:0:1)$, $l_{r+1}'\in (x_0,x_1)^c$ and $(a':b':c')$, $(\alpha':\beta':\gamma')$ are generic in $\pn{2}_\k$.

Hence, to compute the multiplicity $m_z$ of the component supported at the $r$-fold point $z=(0:0:1)$ of $\mb{V}(I^g)$, it is enough to consider that $z$ is the $r$-fold point of a near-pencil of $r+1$ lines.

We treat now the case $r=2$. The near-pencil $\A'=\mb{V}(l_1l_2l_3)$ has three singular point $z,z_2,z_3$ and is always homaloidal, thus:
\[\d{\Phi_{l_1l_2l_3}}{0}=4-m_z-m_2-m_3=1\] and since $m_z,m_2,m_3\geqslant 1$, one has thus $m_z=m_2=m_2=1$. This ends the proof of \Cref{lemmaMultSingLineArr} in the case $r=2$ whether $p=2$ or not.

Consider now the case $r>2$. Put $\A'=\mb{V}(f')$ for the near-pencil of $r+1$ lines:
\begin{enumerate}
\item By \Cref{thmConsGen}, if $p|r$, then \[\d{\Phi_{f'}}{0}=r^2-m_z-\underset{i=1}{\overset{r}{\sum}}m_{z_i}=1\] where $z_1,\ldots,z_n$ are the $r$ singularities defined by the intersection of the $r$ lines of the pencil and the other extra line. These $r$ singularities are all $2$-fold points so, from the case $r=2$, $m_{z_i}=1$ for all $i=1,\ldots,r$ and $m_z=r^2-r-1=(r-1)^2+(r-2)$.
\item If $p\nmid r$, then $\d{\Phi_{f'}}{0}=r-1$ by \Cref{propNaiveDegree} (case $n=r$ of \Cref{propNaiveDegree}). Hence \[\d{\Phi_{f'}}{0}=r^2-m_z-\underset{i=1}{\overset{r}{\sum}}m_{z_i}=r-1\] where $z_1,\ldots,z_n$ are the $r$ singularities defined by the intersection of the $r$ lines of the pencil and the other extra line. As in the previous case $m_{z_i}=1$ for all $i=1,\ldots,r$ so $m_z=(r-1)^2$.
\end{enumerate}
\end{proof}

\begin{proof}[Proof of \Cref{classifLineArr}] As it is stated at the beginning of \cite{Hirzebruch1983ArrLinesAndHyp}, given any line arrangement $\A=\mb{V}(f)$ of $d$ lines, one has the combinatorial identity:
\[ \frac{d(d-1)}{2}=\underset{r=2}{\overset{d}{\sum}}t_r\frac{r(r-1)}{2}\] where $t_r$ is the number of $r$-fold point defined by $\A$. This identity can be re-write as:
\begin{equation}\label{eqCombiCond}(d-1)^2-\underset{r=2}{\overset{d}{\sum}}t_r(r-1)^2-\underset{r=2}{\overset{d}{\sum}}t_r(r-2)=1+(\underset{r=2}{\overset{d}{\sum}}t_r-d)
\end{equation}
Hence, assuming that the characteristic $p$ of $\k$ divides all $r\geqslant 3$ such that $t_r\neq 0$, one has by \Cref{lemmaMultSingLineArr}: \[\d{\Phi_f}{0}=1+(\underset{r=2}{\overset{d}{\sum}}t_r-d).\] Remark that a near-pencils of $d$ lines verifies $\underset{r=2}{\overset{d}{\sum}}t_r=d$ so, if $p$ divides all $r\geqslant 3$ such that $t_r\neq 0$, showing \Cref{classifLineArr} aims to show that the identity $\underset{r=2}{\overset{d}{\sum}}t_r=d$ characterizes near-pencils of $d$ lines among all arrangements of $d$ lines and this fact is established in \cite[Theorem 1]{BruijnErdos1948CombiPb} (see also \cite[Theorem 5.1]{BEUTELSPACHER1995107} for a more recent treatment).

In case $p$ does not divide at least one $r\geqslant 3$ such that $t_r\neq 0$, then by \Cref{labelLemmaMultSing} of \Cref{lemmaMultSingLineArr}, \Cref{eqCombiCond} implies that $\d{\Phi_f}{0}>1$ so $f$ is not homaloidal.
\end{proof}

\bibliographystyle{alpha}

\end{document}